\documentclass[12pt, leqno]{amsart}

\usepackage{pstricks,pst-plot}
\usepackage{mathrsfs}

\usepackage{amsfonts}
\usepackage{amsthm}

\theoremstyle{plain}
\newtheorem{theorem}{Theorem}[section]
\newtheorem{lemma}[theorem]{Lemma}
\newtheorem{corollary}[theorem]{Corollary}

\newcommand{\dz}{dz_1,\ldots, dz_n}
\newcommand{\lspan}{\mathop{\rm span}}

\def \essa {{{\mathscr E}_A}}

\def\wX{\widehat X}

\begin{document}

\title[Analytic Discs]{Analytic Discs and Uniform Algebras Generated by Real-Analytic Functions}
\author{Alexander J. Izzo}
\address{Department of Mathematics and Statistics, Bowling Green State University, Bowling Green, OH 43403}
\email{aizzo@bgsu.edu}


\begin{abstract}
Under very general conditions it is shown$\vphantom{\widehat{\widehat{\widehat{\widehat{\widehat{\widehat{\wX}}}}}}}$ that if $A$ is a uniform algebra generated by real-analytic functions, then either $A$ consists of all continuous functions or else there exists a disc on which every function in $A$ is holomorphic.  This strengthens several earlier results concerning uniform algebras generated by real-analytic functions.

\end{abstract}
\maketitle

\vskip -1.73 true in
\centerline{\footnotesize\it Dedicated to John Wermer on the occasion of his 90th birthday} 
\vskip 1.73 truein

\section{Introduction} \label{intro}

Let
$X$ be a compact Hausdorff space, and let $C(X)$ be the algebra of all continuous complex-valued functions on $X$ with the supremum norm
$ \|f\|_{X} = \sup\{ |f(x)| : x \in X \}$.  A \emph{uniform algebra} $A$ on $X$ is a closed subalgebra of $C(X)$ that contains the constant functions and separates
the points of $X$.  There is a general feeling that a uniform algebra $A$ on $X$ is either $C(X)$ or else there is analytic structure in the maximal ideal space of $A$ with respect to which the Gelfand transforms of the functions in $A$ are holomorphic.  Much of the theory of uniform algebras is motivated by this point of view.  It is well known though, that this feeling is not completely correct.  However, in this paper we will prove that such a dichotomy 
is literally true in the case of a uniform algebra generated by 
real-analytic functions on its maximal ideal space.  The precise statement of the result is as follows.

\begin{theorem} \label{main_theorem} Let $V$ be a real-analytic subvariety of an open set $\Omega \subset \mathbb{R}^n$, and let $K$ be a compact subset of $V$ such that $\partial K$ is a real-analytic subvariety of $V$.  Let $A$ be a uniform algebra on $K$ generated by a collection $\Phi$ of functions real-analytic on $K$.  Suppose also that the maximal ideal space of $A$ is $K$.  Then either $A = C(K)$ or else $K$ contains an analytic disc.
\end{theorem}

A few points in the statement of the theorem should be clarified.
Throughout the paper, given a real-analytic subvariety $V$ of an open set $\Omega \subset \mathbb{R}^n$, and given a subset $K$ of $V$, whenever we say 
\lq\lq a collection $\Phi$ of functions real-analytic on $K$,\rq\rq\ we mean that to each member $f$ of $\Phi$ there corresponds a neighborhood of $K$ in $\mathbb{R}^n$, that may depend on the function $f$, to which $f$ extends to be real-analytic.  To say that $K$ is the maximal ideal space of $A$ means that
every non-zero multiplicative linear functional on $A$ is given by point evaluation at some point of $K$.   Let $D$ denote the open unit disc in the complex plane. To say that $K$ contains an analytic disc means that
there is a continuous injective map $\sigma :D \to K$ such that $f \circ \sigma$ is holomorphic for every $f \in A$.

The above theorem strengthens earlier work of the author and others motivated by the so called peak point conjecture.  We recall some of the work related to this conjecture here.  For more complete discussions see \cite{anderson_izzo_varieties} and \cite{izzo_survey}.  Consider the following two conditions on a uniform algebra $A$ on a compact metric space $K$:
\begin{description}
\item [{\rm(i)}]  the maximal ideal space of $A$ is $K$,
\item [{\rm(ii)}]  each point of $K$ is a peak point for $A$, i.e., given $x \in K$ there exists $f \in A$ with $f(x) = 1$ and $|f(y)| < 1$ for all $y \in K \setminus \{x\}$.                     \end{description}
Both (i) and (ii) are necessary conditions for $A = C(K)$.  It was once conjectured that (i) and (ii) together were also sufficient to conclude that $A = C(K)$.  A counterexample to this peak point conjecture was constructed by Brian Cole in 1968  \cite{Co} (or see \cite[Appendix]{browder}, or \cite[Section~19]{stout_book}).
Several other counterexamples have been found since then. 
Nevertheless, the peak point conjecture has been shown to be true in a number of special cases.
Anderson and the present author \cite{anderson_izzo_two_manifolds} proved that if $K$ is a compact differentiable two-dimensional manifold-with-boundary, and $A$ is a uniform algebra generated by $C^1$-smooth functions that satisfies (i) and (ii), then $A = C(K)$.  An example of Basener \cite{basener} on the three-sphere shows that the corresponding statement for three-manifolds is false.
However, Anderson, Wermer, and the present author \cite{anderson_izzo_smooth_manifolds}, \cite{anderson_izzo_varieties}, \cite{anderson_izzo_wermer_three_dimensions}, \cite{anderson_izzo_wermer_varieties} established peak point theorems for uniform algebras generated by real-analytic functions on real-analytic varieties.  The latest result along those lines is the following.

\begin{theorem}[\cite{anderson_izzo_varieties}, Theorem~1.1]  Let $V$ be a real-analytic subvariety of an open set $\Omega \subset \mathbb{R}^n$, and let $K$ be a compact subset of $V$ such that $\partial K$ is a real-analytic subvariety of $V$.  Let $A$ be a uniform algebra on $K$ generated by a collection $\Phi$ of functions real-analytic on $K$.  Assume that $A$ satisfies conditions (i) and (ii) above.  Then $A = C(K)$.
\end{theorem}

Note that this result is an immediate consequence of Theorem~\ref{main_theorem}.

Lee Stout \cite{stout_varieties} proved a result concerning approximation on real-analytic varieties without the peak point hypothesis.

\begin{theorem}[\cite{stout_varieties}, Theorem~1]
If $K$ is a compact, real-analytic subvariety of $\mathbb{C}^n$ that is holomorphically convex, then every continuous function on $K$ can be approximated uniformly by functions holomorphic on {\rm(}varying neighborhoods of\/{\rm)} $K$.  {\rm(}Here the hypothesis that $K$ is holomorphically convex means that the maximal ideal space of the algebra $\mathcal{O}(K)$ of functions holomorphic on $K$ is $K$.{\rm)} 
\end{theorem}

A key ingredient in Stout's proof is 
a theorem of Diederich and Fornaess \cite{diederich-fornaess}, which states that a compact real-analytic variety in $\mathbb{C}^n$ contains no non-trivial germ of a complex-analytic variety. Observe that in view of this result of Diederich and Fornaess, Stout's theorem follows from Theorem~\ref{main_theorem} above.

In \cite{anderson_izzo_local} Anderson and the present author used an argument similar to the one used to prove the peak point theorem in \cite{anderson_izzo_varieties} to prove a localization result for uniform algebras generated by real-analytic functions.  That result also follows from Theorem~\ref{main_theorem}.  In stating the result we use the following terminology: A uniform algebra $A$ on a compact space $X$ is said to have the \emph{countable approximation property} if for each $f \in C(X)$ there exists a countable collection $\{M_{n}\}$ of compact subsets of $X$ with $\cup_{n=1}^{\infty} M_{n} = X$ and $f\big|_{M_{n}} \in \overline{A\big|_{M_{n}}}$ for each $n$.

\begin{theorem} Let $V$ be a real-analytic subvariety of an open set $\Omega \subset \mathbb{R}^n$, and let $K$ be a compact subset of $V$ such that $\partial K$ is a real-analytic subvariety of $V$.  Let $A$ be a uniform algebra on $K$ generated by a collection $\Phi$ of functions real-analytic on $K$.  Assume that the maximal ideal space of $A$ is $K$, and that $A$ has the countable approximation property. Then $A = C(K)$.
\end{theorem}

To show that this follows from Theorem~\ref{main_theorem} it suffices to show that if a uniform algebra $A$ on a compact set $K$ satisfies the countable approximation property, then $K$ contains no analytic disc.  Suppose on the contrary, that there is an analytic disc in $K$, i.e., a continuous injective map $\sigma:D\rightarrow K$ such that $f\circ \sigma$ is holomorphic for every function $f$ in $A$.  By replacing $D$ by a smaller disc with compact closure in $D$, we may assume that $\sigma$ extends to a homeomorphism of $\overline D$ into $K$.  Then there exists a function $g\in C(K)$ such that $g\circ\sigma=\overline z$.  Let $\{M_n\}$ be a countable collection of compact subsets of $K$ with $\cup_{n=1}^{\infty} M_{n} = K$.  By the Baire category theorem, some $M_n$ contains the image under $\sigma$ of an open set of $D$.  Then $g|M_n\notin \overline{A|M_n}$.  Thus $A$ does not have the countable approximation property.

It should be pointed out that the two-dimensional peak point theorem in 
\cite{anderson_izzo_two_manifolds} mentioned above is of a different nature from the theorems involving real-analytic functions since it requires only $C^1$-smoothness.  That theorem is not contained in Theorem~\ref{main_theorem}.  It has, however, been strengthened by Swarup Ghosh \cite{ghosh} by replacing the peak point hypothesis by the the hypothesis that each point of the maximal ideal space is isolated in the topology inherited from the normed dual $A^*$ of $A$.  The real-analytic peak point theorems in dimension three given in \cite{anderson_izzo_smooth_manifolds} and 
\cite{anderson_izzo_wermer_three_dimensions} have also been similarly strengthened by Ghosh \cite{ghosh}, \cite{ghosh2}.

Much of the proof of Theorem~\ref{main_theorem} is essentially a repetition of the proofs given in \cite{anderson_izzo_varieties} and \cite{anderson_izzo_local}.  We will, nevertheless, present the proof in full detail as we believe that the proof should be given explicitly in the literature, and changes from the earlier  proofs are needed.
In particular different lemmas must be used.  Moreover, while the proofs in \cite{anderson_izzo_varieties} and \cite{anderson_izzo_local} were carried out via duality working with an individual annihilating measure for the algebra $A$, it seems that that approach does not work here and that we must instead consider all annihilating measures simultaneously.  That can be done conveniently by using the notion of essential set which was introduced by Herbert Bear \cite{bear} and whose definite is recalled below in the next section.  In fact, by using this notion, we will avoid altogether explicit mention of annihilating measures in the proof.  

To prove Theorem~\ref{main_theorem}, we will prove the following.

\begin{theorem} \label{support_theorem} Let $V$ be a real-analytic subvariety of an open set $\Omega \subset \mathbb{R}^n$, and let $K$ be a compact subset of $V$.  Let $A$ be a uniform algebra on $K$ generated by a collection $\Phi$ of functions real-analytic on $K$.  Suppose that the maximal ideal space of $A$ is $K$ and that $K$ contains no analytic discs.  Then the essential set $\essa$ for $A$ is contained in $\partial K$. \end{theorem}

Here and throughout the paper $\partial K$ denotes the boundary of $K$ relative to $V$.  The interior of $K$ relative to $V$ will be denoted by ${\rm int}(K)$.

We will show at the beginning of Section~3 that Theorem~\ref{main_theorem} is an easy consequence of Theorem~\ref{support_theorem}, whose proof occupies the bulk of Section~3.
In Section \ref{prelims} we collect some preliminary lemmas and comment on the general outline of the proof of Theorem \ref{support_theorem}.

While Theorem~\ref{main_theorem} seems to be the definitive result concerning conditions under which a uniform algebra generated by real-analytic functions on its maximal ideal space consists of all continuous functions, the theorem suggests the possibility of an even stronger theorem.  Specifically, it suggests the possibility of completely describing uniform algebras generated by real-analytic functions in terms of CR structures.  That possibility will, however, not be explored here.

It is a pleasure to dedicate this paper to John Wermer on the
occasion of his 90th birthday.  Much of the author's work has been inspired by Wermer, and that is especially so in the present case.  Indeed, the main result presented here can be seen as the end result of a joint project of Anderson, Wermer, and the author that was initiated in 1999 by Wermer's suggestion to seek a peak point theorem for uniform algebras generated by real-analytic functions.

\section{Preliminaries} \label{prelims}

Let $\Sigma$ be a real-analytic variety in the open set $W \subset \mathbb{R}^n$, that is, for each point $p \in W$ there is a finite set $\mathcal{F}$ of functions real-analytic in a neighborhood $N$ of $p$ in $W$ with $\Sigma $ the common zero set of $\mathcal{F}$ in $N$. The following definitions and facts are standard; see for example \cite{narasimhan}. We let $\Sigma_{\mbox{\tiny reg}}$ denote the set of points $p \in \Sigma$ for which there exists a neighborhood $N$ of $p$ in $\mathbb{R}^n$ such that $\Sigma \cap N$ is a regularly
imbedded real-analytic submanifold of $N$ of some dimension $d = d(p)$.   This dimension $d(p)$ is locally constant on $\Sigma_{\mbox{\tiny reg}}$.  The dimension of $\Sigma,$ denoted by $\mbox{dim}(\Sigma)$, is defined to be the largest such $d(p)$ as $p$ ranges over the regular points of $\Sigma$.  The singular set of $\Sigma$, denoted by $\Sigma _{\mbox{\tiny sing}}$, is the complement in $\Sigma$ of $\Sigma_{\mbox{\tiny reg}}$.  If $\Sigma^{\prime} \subset \Sigma$ is a real-analytic subvariety of some open set $W^{\prime} \subset W$ and $\Sigma^{\prime}$ has empty interior relative to $\Sigma$, then $\mbox{dim}(\Sigma^{\prime}) < \mbox{dim}(\Sigma)$.  Both $\Sigma$ and $\Sigma_{\mbox{\tiny sing}}$ are closed in $W$.  Although $\Sigma_{\mbox{\tiny sing}}$ may not itself be a subvariety of $W$, it is locally contained in a proper subvariety of $\Sigma$: for each $p \in \Sigma _{\mbox{\tiny sing}}$ there is a real-analytic subvariety $Y$ of an open neighborhood $N$ of $p$ such that $\Sigma_{\mbox{\tiny sing}} \cap N \subset Y$  and $\mbox{dim}(Y) < \mbox{dim}(\Sigma)$.

If $L$ is a subset of an $m$-dimensional manifold $M$ of class $C^1$ and
$\Phi$ is a collection 
of functions that are $C^{1}$ on (varying neighborhoods of) 
$L$, we define the {\em exceptional set $L_{\Phi}$ of $L$ relative to $\Phi$}  by
\begin{equation} L_{\Phi} = \{ p \in L : df_{1} \wedge \ldots \wedge df_{m}(p) = 0 \mbox{ for all }m \mbox{-tuples } (f_{1}, \ldots , f_{m}) \in  \Phi^m \}. \label{eq:exceptional} \end{equation}
If $\Sigma$ is a real-analytic variety and $\Phi$ is a collection of functions real-analytic on $\Sigma$ (i.e., each function in $\Phi$ extends to be real-analytic in a neighborhood, depending on the function, of $\Sigma$), then $\Sigma_{\Phi}$ is defined to be the set of all points $p \in \Sigma_{\mbox{\tiny reg}}$ such that in a neighborhood of $p$, the set  $\Sigma_{\mbox{\tiny reg}}$ is an $m$-dimensional manifold $M$, and $p \in M_{\Phi}$ as defined in (\ref{eq:exceptional}).

The following is proven in \cite[Lemma~2.1]{anderson_izzo_varieties}.

\begin{lemma} \label{exceptional_set_variety} Let $\Sigma$ be a real-analytic variety in an open set $W \subset \mathbb{R}^n$, and $\Phi$ a collection of functions real-analytic on $\Sigma$.  Then $\Sigma_{\Phi}$ is a subvariety of $W \setminus \Sigma _{\mbox{\rm\tiny sing}}$. \end{lemma}

The next lemma (a generalization of \cite[Lemma~2.5]{anderson_izzo_wermer_three_dimensions}) and the corollary that follows are no doubt well known, but they do not seem to be presented in the standard texts.  Here we denote the dimension of a vector space $V$ over a field $F$ by $\dim_F V$.

\begin{lemma}\label{dim-complex-tangent-space}
Let $M$ be a real $m$-dimensional submanifold of $\mathbb{C}^n$ of class $C^1$.  Regard the differentials $\dz$ of the coordinate functions $z_1,\ldots z_n$ as forms on $M$.  Then the dimension of the complex tangent space $H_pM$ to $M$ at a point $p$ is given by
$$\dim_{\mathbb{C}} H_pM=m-\dim_{\mathbb{C}}\lspan\{\dz\}.$$
\end{lemma}

\begin{proof}
Fix $p\in M$.  Note that
\begin{eqnarray*}
\dim_{\mathbb{R}}(T_pM+iT_pM)+\dim_{\mathbb{R}}H_pM&=&
\dim_{\mathbb{R}}(T_pM+iT_pM)+\dim_{\mathbb{R}}(T_pM\cap iT_pM)\\
&=&2m.
\end{eqnarray*}
Thus it suffices to show that 
\begin{equation}\label{dim-equation}
\dim_{\mathbb{C}}(T_pM+iT_pM)=\dim_{\mathbb{C}}\lspan\{\dz\}.
\end{equation}

Choose local coordinates $u_1,\ldots, u_m$ on $M$ in a neighborhood of $p$.  The vectors\break $\bigl(\partial z_1/\partial u_k (p),\ldots, \partial z_n/ \partial u_k(p)\bigr)$, $1\leq k\leq m$, form a basis for $T_pM+iT_pM$ over $\mathbb{C}$.
The forms $\dz$ can be expressed in terms of the basis $du_1,\ldots, du_m$ for the complexified cotangent space as 
$$dz_j=\frac{\partial z_j}{\partial u_1}\,du_1 +\cdots+\frac{\partial z_j}{\partial u_m} \,du_m.$$
Thus both sides of equation (\ref{dim-equation}) equal the rank of the matrix $(\partial z_j/\partial u_k)$ at $p$.
\end{proof}

\begin{corollary}\label{dim-cor}
Let $M$ be a real $m$-dimensional manifold of class $C^1$.  Let $f_1,\ldots, f_m$ be $C^1$-smooth functions on a neighborhood $U$ of a point $x\in M$ such that $f=(f_1,\ldots, f_m):U\rightarrow \mathbb{C}^m$ is an embedding.  Then the dimension of the complex tangent space $H_{f(x)}f(U)$ to $f(U)$ at $f(x)$ is given by
$$\dim_{\mathbb{C}}H_{f(x)}f(U)=m-\dim_{\mathbb{C}}\lspan\{df_1,\ldots, df_m\}.$$
\end{corollary}

\begin{proof}
This follows immediately from the preceding lemma using the pull back
$f^*:T_{f(x)}f(U)\rightarrow T_xU$.
\end{proof}

\begin{lemma}\label{exceptional_set_empty_int} Let $M$ be a real $m$-dimensional manifold of class $C^{2}$.
Suppose $K$ is a compact subset of $M$, and $A$ is a uniform algebra on $K$ generated by a collection $\Phi$ of functions of class $C^2$ on (varying neighborhoods of) $K$ in $M$.  If the exceptional set $K_{\Phi}$ has nonempty interior ${\rm int}(K_{\Phi})$ in $M$, then the maximal ideal space of $A$ contains an analytic disc whose boundary is contained in ${\rm int}(K_{\Phi})$.
\end{lemma}

The proof will use ideas contained in the proof of \cite[Theorem~1.13]{izzo_kalm_wold} along with the following lemma which follows from a result of Errett Bishop \cite[Theorem~18.7]{Alex-Wermer} and is proven in \cite[Lemma~3.8]{ghosh}.

\begin{lemma}\label{swarup-lemma} Let $M$ be a real $m$-dimensional submanifold of $\mathbb{C}^n$ of class $C^{2}$.  Let $E$ be the set of points at which $M$ has a complex tangent.  Assume that $U$ is an open subset of $\mathbb{C}^n$ so that $M\cap U$ is a nonempty subset of $E$.  Then $M\cap U$ 
contains the boundary of an analytic disc.\end{lemma}

\begin{proof}[Proof of Lemma~\ref{exceptional_set_empty_int}]
Suppose that $K_\Phi$ has nonempty interior in $M$.  Let $r$ be the largest integer such that for some point $p$ in the interior of $K_\Phi$ in $M$ there exist $r$ functions $f_1,\ldots, f_r\in \Phi$ with $df_1\wedge\cdots\wedge df_r(p)\neq 0$.  Then $df_1\wedge\cdots\wedge df_r\neq 0$ throughout a neighborhood $\Omega$ of $p$ in ${\rm int}(K_{\Phi})$.

We claim that the differentials of the real and imaginary parts of the functions in $\Phi$ span the cotangent space to $M$ at some point $x$ in $\Omega$ (and hence at all points in some neighborhood of $x$).  To see this, assume that it is false, and let $k$ denote the maximum dimension of the spaces $\lspan \{du(y): u {\rm\ is\ the\ real\ or\ imaginary\ part\ of\ a\ function\ in\ } \Phi\}$ over all points $y\in \Omega$.  Fix a point $y\in \Omega$ where this maximum is achieved and choose functions $u_1,\ldots, u_k$, with each $u_j$ the real or imaginary part of a function $g_j$ in $\Phi$, such that $du_1\wedge \cdots \wedge du_k(y)\neq 0$.  Then in a neighborhood of $y$, the common level set of $g_1,\ldots, g_k$ passing through $y$ is a $C^1$-smooth manifold of positive dimension on which every function in $\Phi$ is constant, contrary to the fact that $\Phi$ separates points on $K$.  This establishes the claim.  Thus we can choose functions $f_{r+1},\ldots, f_n$ in $\Phi$ such that the differentials of the real and imaginary parts of $f_1,\ldots, f_n$ span the cotangent space to $M$ at all points in some neighborhood $U\subset\Omega$ of a point $x$ in $\Omega$.  Then the map $f=(f_1,\ldots, f_n)|U:U\rightarrow \mathbb{C}^n$ is an embedding.  

At every point of $U$ the dimension of the complex linear span of $df_1,\ldots, df_n$ is equal to $r<m$.  Consequently, $f(U)$ is a CR manifold with complex tangent space of dimension $m-r>0$ by Corollary~\ref{dim-cor}.  Note that by our choice of $f_1,\ldots, f_r$, we have that for every function $g\in \Phi$, the differential $dg(y)$ belongs to $\lspan\{df_1(y),\ldots, df_n(y)\}$ for every point $y\in U$.  It follows that regarding functions in $A$ as functions on $U$ (by precomposing with $f^{-1}$), every function in $\Phi$ is CR on $f(U)$.  Therefore, the approximation theorem of Salah Baouendi and Fran\c cois Treves \cite[Ch.~13, Theorem~1]{boggess} gives that, shrinking $U$ if necessary, each function $g\in \Phi$ can be approximated uniformly on $f(U)$ by polynomials.  
(In \cite{boggess} the approximation theorem is stated only for generic CR manifolds.  However, the theorem for nongeneric CR manifolds follows from the generic case.)
Thus for $N$ a closed neighborhood of $x$ contained in $U$, the restriction algebra $\overline{A|N}$ can be identified with the uniform closure of the polynomials on $f(N)$.  Since the complex tangent space to $f(U)$ is everywhere of dimension $m-r>0$, Lemma~\ref{swarup-lemma} gives that $f(N)$ contains the boundary of an analytic disc in $\mathbb{C}^n$.  Finally, since the maximal ideal space of $\overline{A|N}$ is contained in the maximal ideal space of $A$, the desired conclusion follows.
\end{proof}

The following result of the author \cite{izzo_approx_on_manifolds} will enable us to reduce approximation on a variety to approximation on the union of the exceptional set and the singular set of the variety.  This type of theorem has a long history, going back to work of John Wermer \cite{wermer64} and \cite{wermer} and Michael Freeman \cite{freeman} in the 1960's - for a detailed account, see \cite{izzo_approx_on_manifolds}.

\begin{theorem}[\cite{izzo_approx_on_manifolds},~Theorem~1.3] \label{approximation_theorem}
Let $A$ be a uniform algebra on a compact
Hausdorff space $X$, and suppose that the maximal ideal space of $A$ is $X$.
Suppose also that $E$ is a closed subset of $X$ such that $X \setminus E$
is an $m$-dimensional manifold and such that
\begin{enumerate}
\item for each point $p \in X \setminus E$ there are functions
$f_1, \ldots, f_m$ in $A$ that are $C^1$ on $X \setminus E$ and satisfy $df_1
\wedge \ldots \wedge df_m (p) \neq 0$, and
\item  the functions in $A$ that are $C^1$ on $X \setminus E$ separate points
on $X$.
  \end{enumerate}
Then $A= \{g \in C(X) : g | E \in A|E \}$.  \end{theorem}

The general idea of the proof of Theorem \ref{support_theorem} is to use Theorem \ref{approximation_theorem} to reduce approximation on a variety $V$ to approximation on successively smaller and smaller sets.  In \cite{anderson_izzo_varieties} and \cite{anderson_izzo_local} this was done by considering the support of an arbitrary annihilating measure.  As noted in the introduction, it seems that in the present setting dealing with an individual annihilating measure does not work and that we must, in effect, consider all the annihilating measures at once.  This will be done indirectly using the notion of essential set due to Bear \cite{bear}.  The {\it essential set} for a uniform
algebra $A$ on a space $X$ is the unique smallest closed subset $\essa$ of $X$ such
that $A$ contains every continuous function on $X$ that vanishes on
$\essa$.   For a proof of the existence of the essential set and other details, see \cite[pp. 144--147]{browder}.
Note that $A$ contains every continuous function whose
restriction to $\essa$ lies in the restriction of $A$ to $\essa$.
An alternative description of the essential set $\essa$ is that $\essa$ is the closure of the union of the supports of all the annihilating measures for $A$.  For every closed subset $Y$ of $X$ containing $\essa$, the algebra $A|Y$ obtained by restricting the functions in $A$ to $Y$ is closed in $C(Y)$.  Note that $A=C(X)$ if and only if $\essa$ is the empty set.  Note also that the conclusion of Theorem~\ref{approximation_theorem} can be rephrased as the statement that the essential set $\essa$ for $A$ is contained in $E$.

We will need the following two lemmas regarding the essential set.

\begin{lemma}\label{inheritance_lemma} Let $A$ be a uniform algebra on a compact Hausdorff space $X$ and suppose that the maximal ideal space of $A$ is $X$.  Let $Y$ be a closed subset
of $X$ that contains the essential set for $A$.   Then the maximal ideal space of  $A|Y$ is $Y$. 
\end{lemma}

\begin{proof}
This is well known and easily proved as follows: The maximal ideal space of 
$A|Y$ consists of those points $p$ of $X$ such that $|f(p)|\leq \|f\|_Y$ for every $f\in A$ (see \cite[II.6]{gamelin}), and because $A$ contains every continuous function on $X$ that vanishes on $Y$, the points of $Y$ are the only ones satisfying this condition.
\end{proof}

\begin{lemma}\label{essential}
Let $A$ be a uniform algebra on a compact Hausdorff space $X$, and suppose that the maximal ideal space of $A$ is $X$.  Then the essential set $\essa$ for $A$ has no isolated points.
\end{lemma}

\begin{proof}
This follows easily from \cite[Theorem~6]{bear}, and is easily proved in a similar fashion as follows:
By the preceding lemma, the maximal ideal space of $A|\essa$ is $\essa$.  Consequently, if $\essa$ has an isolated point $p$, then the Shilov idempotent theorem \cite[Theorem~III.6.5]{gamelin} shows that $A|\essa$ contains every function on $\essa$ that vanishes identically on $\essa\setminus\{p\}$.  This contradicts the hypothesis that $\essa$ is the essential set for $A$.
\end{proof}

To prove Theorem \ref{support_theorem} we use Theorem \ref{approximation_theorem} to show that the essential set $\essa$ is contained in the union of the singular set of $V$ and the exceptional set of the algebra $A$.  As we have noted, the singular set of $V$ is locally contained in a proper subvariety of $V$, i.e., a variety of dimension strictly less than the dimension of $V$.  The exceptional set, in the absence of analytic discs, is also a proper subvariety of the regular set of $V$, by Lemma~\ref{exceptional_set_variety} combined with Lemma~\ref{exceptional_set_empty_int}.  One would like to then use induction to show that $\essa$ is contained in each of a sequence of varieties of decreasing dimension, until the dimension is zero (i.e., the variety is a discrete set), and then to conclude that $\essa$ is empty by invoking Lemma~\ref{essential}.  However, it is not obvious that the \emph{union} of the exceptional set and the singular set, even locally, must itself be contained in a subvariety of 
$V$ of dimension less than that of $V$.
To get around this difficulty, we proceed as in \cite{anderson_izzo_varieties} and \cite{anderson_izzo_local} treating the exceptional set and singular set separately, introducing a filtration of $V$ into exceptional sets and singular sets of decreasing dimensions.   We show by induction on decreasing dimension of the exceptional sets that $\essa$ must lie 
in the singular set.  We then use induction again on a decreasing sequence of singular sets to reduce $\essa$ to the empty set.


\section{Proof of Theorems~\ref{main_theorem} and~\ref{support_theorem}} \label{proof}

We first indicate how 
Theorem \ref{main_theorem} follows from
Theorem~\ref{support_theorem}.
Let the variety $V$, the compact set $K \subset V$, and the algebra $A$ be as in Theorem~\ref{main_theorem}.  Suppose also that $K$ contains no analytic discs.  Theorem~\ref{support_theorem} gives that the essential set $\essa$ for $A$ is contained in $\partial K$.  Then 
Lemma~\ref{inheritance_lemma} gives that the maximal ideal space of $A|\partial K$ is $\partial K$.  Note that the boundary of $\partial K$ relative to itself is empty.  Therefore, we can apply Theorem~\ref{support_theorem} with $V$ replaced by $\partial K$ and $K$ replaced by $\partial K$ also to conclude that the essential set for $A|\partial K$ is empty.  Consequently, $A=C(K)$, and Theorem~\ref{main_theorem} is established.

We now turn to the proof of Theorem \ref{support_theorem}, beginning with a general construction from \cite{anderson_izzo_varieties}.

Let $\Sigma$ be a real-analytic variety in the open set $W \subset \mathbb{R}^n$, and let $\Phi$ be a collection of functions real-analytic on $\Sigma$.  We define inductively subsets $\Sigma_{k}$ of $\Sigma$ such that $\Sigma_{0} = \Sigma$, and for $k \geq 1$, $\Sigma_{k}$ is a real-analytic subvariety of
\[ W_{k} := W \setminus \bigcup_{j=0}^{k-1} \; (\Sigma_{j})_{\mbox{\tiny sing}} \]
defined by
\[ \Sigma_{k} = (\Sigma_{k-1})_{\Phi}. \]
Note that by definition, $\Sigma_{k} \subset (\Sigma_{k-1})_{\mbox{\tiny reg}}$.
We will refer to the varieties $\Sigma_{k}$ as the {\em E-filtration of $\Sigma$ in $W$ with respect to $\Phi$}, and to the sets $(\Sigma_{k})_{\mbox{\tiny sing}}$ as the {\em S-filtration of $\Sigma$ in $W$} (\emph{E} for exceptional, \emph{S} for singular).

\begin{lemma} \label{dimension_lemma} With $V,\Omega,K, A, \Phi$ as in Theorem \ref{support_theorem}, suppose that $W$ is an open subset of $\Omega$  and $\Sigma \subset {\rm int}(K) \cap W$ is a real-analytic subvariety of $W$.  Let $\{ \Sigma_{k} \}$ be the E-filtration of $\Sigma$ in $W$ with respect to $\Phi$.  Then for each $k$, the dimension of $\Sigma_{k}$ is no more than $\mbox{dim} (\Sigma) - k$. \end{lemma}

\begin{proof} Let $d = \mbox{dim}(\Sigma)$. The proof is by induction on $k$. The result is clear when $k = 0$.  Suppose we have shown for some $k$ that $\mbox{dim}(\Sigma_{k})\leq d - k$.  Fix $p \in (\Sigma_{k})_{\mbox{\tiny reg}}$, and let $U$ be a smoothly bounded
neighborhood of $p$ in $(\Sigma_{k})_{\mbox{\tiny reg}}$ with $\overline{U} \subset (\Sigma_{k})_{\mbox{\tiny reg}}$.  We may assume that
$U$ has constant dimension (by induction, no more than $d - k$) as a submanifold of
$\mathbb{R}^n$.  Applying
Lemma~\ref{exceptional_set_empty_int}, taking $M$ to be a neighborhood of $\overline U$ in $(\Sigma_{k})_{\mbox{\tiny reg}}$ and
replacing $A$ by $\overline{A|{\overline{U}}}$ yields that $\Sigma_{k+1} = (\Sigma_{k})_{\Phi}$ has no interior in $U$.  Since $p$ was arbitrary, we conclude that
\[ \mbox{dim}(\Sigma_{k+1}) \leq \mbox{dim}(\Sigma_{k}) - 1 \leq d - (k+1). \]
By induction, the proof is complete. \end{proof}

Note that Lemma \ref{dimension_lemma} implies, with $d = \mbox{dim}(\Sigma)$,  that $\Sigma_{d}$ is a zero-dimensional variety, i.e., a discrete set.

Let $B(p,r)$ denote the open ball of radius $r$ centered at $p \in \mathbb{C}^n$.

\begin{lemma} \label{support_in_S_filtration} With $V,\Omega,K, A, \Phi$ as in Theorem \ref{support_theorem}, assume $p \in {\rm int}(K)$.  If $r > 0$ is such that $V \cap B(p,r) \subset {\rm int}(K)$, and there is a real-analytic
$d$-dimensional subvariety $\Sigma \subset V$ of $B(p,r)$ with\/ $\essa \cap B(p,r) \subset \Sigma$, then\/ $\essa \cap B(p,r)$ is contained in the S--filtration of $\Sigma$, i.e., $\essa \cap B(p,r)\subset
\bigcup\limits_{k=0}^{d-1} (\Sigma_{k})_{\mbox{\rm\tiny sing}}$.
\end{lemma}

\begin{proof}
We will show by induction on $J$ that
\[ \essa \cap B(p,r) \subset \bigcup_{k=0}^{J} (\Sigma_{k})_{\mbox{\tiny sing}} \cup \Sigma_{J+1} \]
for each $J$, $0 \leq J \leq d-1$.
The desired conclusion then follows from Lemma~\ref{essential} since $\Sigma_d$ is a discrete set.

For the $J = 0$ case, let $X = (K \setminus B(p,r)) \cup \Sigma_{0} $ and let $E =(K \setminus B(p,r)) \cup (\Sigma_{0})_{\mbox{\tiny sing}} \cup \Sigma_{1}$.  Note that both $X$ and $E$ are closed.  We want to show that $\essa\subset E$.   
By hypothesis, $\essa\subset X$.  Therefore,
Lemma~\ref{inheritance_lemma} gives that the maximal ideal space of
${A|X}$ is $X$.  Note that $X \setminus E = \Sigma_{0} \setminus \bigl((\Sigma_{0})_{\mbox{\tiny sing}} \cup \Sigma_{1}\bigr)$ satisfies the hypotheses of Theorem~\ref{approximation_theorem}.  Therefore by Theorem~\ref{approximation_theorem}, if $g \in C(K)$ vanishes identically on $E$, then $g|X$ belongs to ${A|X}$.  Since by hypothesis $\essa\subset X$, we get that each $g\in C(K)$ vanishing identically on $E$ belongs to $A$.  Thus $\essa\subset E$, as desired.

The general induction step is similar: assuming the result for some $0\leq J < d-1$, we set
\[ X = (K \setminus B(p,r)) \cup \bigcup_{k=0}^{J} (\Sigma_{k})_{\mbox{\tiny sing}} \cup \Sigma_{J+1}, \]
\[ E = (K \setminus B(p,r)) \cup \bigcup_{k=0}^{J+1} (\Sigma_{k})_{\mbox{\tiny sing}} \cup \Sigma_{J+2}, \]
Both $X$ and $E$ are closed.  Noting that the induction hypothesis implies that $\essa \subset X$, we apply Theorem \ref{approximation_theorem} to ${A|X}$ as above, to conclude that $\essa \subset E$.  
\end{proof}

\begin{lemma} \label{support_in_boundary} With $V,\Omega,K, A, \Phi$ as in Theorem \ref{support_theorem}, assume $p \in {\rm int}(K)$.  Assume also that there exist $r > 0$ with $B(p,r) \cap V \subset {\rm int}(K)$ and a real-analytic subvariety $\Sigma \subset V$ of $B(p,r)$ with\/ $\essa \cap B(p,r)$ contained in the S--filtration of $\Sigma$ in $B(p,r)$.  Then there exists $r^{\prime} > 0$ such that\/ $\essa \cap B(p,r^{\prime}) = \emptyset$. \end{lemma}

\begin{proof} We apply induction on the dimension of $\Sigma$.  If $\mbox{dim}(\Sigma) = 0$, then $\Sigma$ is discrete, and Lemma~\ref{essential} shows that $|\mu|(\Sigma)=0$.  Now suppose the conclusion of the Lemma holds whenever $\mbox{dim}(\Sigma) < d$.  If $\mbox{dim}(\Sigma) = d$, let $\Sigma_{0}, \ldots, \Sigma_{d}$ be the \emph{E}-filtration of $\Sigma$ in $B(p,r)$.  (Recall that $\Sigma_{d}$ is discrete.) By induction on $J$ we will show that
\begin{equation} \essa \cap B(p,r) \subset \bigcup_{k=0}^{d-1-J} (\Sigma_{k})_{\mbox{\tiny sing}}. \label{eq:supp} \end{equation}
for $J = 0, \ldots , d-1$.  The case $J = 0$ is the hypothesis of the Lemma.  Assume we have established (\ref{eq:supp}) for some $J$, $0 \leq J < d-1$.
To show that (\ref{eq:supp}) holds with $J$ replaced by $J + 1$, we must show that $\essa \cap (\Sigma_{d-1-J})_{\mbox{\tiny sing}} = \emptyset$.  Fix $q \in (\Sigma_{d-1-J})_{\mbox{\tiny sing}}$.  By construction of the $\Sigma_{k}$, there exists $s > 0$ so that $B(q,s) \subset B(p,r)$ and $(\Sigma_{k})_{\mbox{\tiny sing}} \cap B(q,s) = \emptyset$ for all $k < d-1-J$. Therefore, the induction hypothesis implies that $\essa \cap B(q,s) \subset (\Sigma_{d-1-J})_{\mbox{\tiny sing}}$.  Replacing $s$ by a smaller positive number if necessary, we may assume that there is a real-analytic subvariety $Y$ of $B(q,s)$ with $(\Sigma_{d-1-J})_{\mbox{\tiny sing}} \subset Y \subset V$ and $\mbox{dim}(Y) < \mbox{dim}(\Sigma_{d-1-J}) \leq J + 1 < d$ (the next-to-last inequality following from Lemma \ref{dimension_lemma}).  By Lemma \ref{support_in_S_filtration}, $\essa \cap B(q,s)$ is contained in the \emph{S}-filtration of $Y$ in $B(q,s)$.  Now note that our induction hypothesis on dimension implies that the conclusion of Lemma \ref{support_in_boundary} holds with $\Sigma$ replaced by $Y$, since $\mbox{dim}(Y) < d$.  We conclude that there exists $s'>0$ such that $\essa \cap B(q,s') = \emptyset$.  Since $q  \in (\Sigma_{d-1-J})_{\mbox{\tiny sing}}$ was arbitrary, this shows that $\essa \cap (\Sigma_{d-1-J})_{\mbox{\tiny sing}} = \emptyset$ and completes the proof that (\ref{eq:supp}) holds for $J = 0, \ldots , d-1$.

Finally, the case $J = d-1$ of (\ref{eq:supp}) asserts that $\essa \cap B(p,r) \subset (\Sigma_{0})_{\mbox{\tiny sing}}$.   We may choose $t$ with $0 < t< r$ and a subvariety $Y \subset V$ of $B(p,t)$ so that $(\Sigma_{0})_{\mbox{\tiny sing}} \cap B(p,t)\subset Y$ and $\mbox{dim}(Y) < \mbox{dim}(\Sigma) = d$.  By Lemma \ref{support_in_S_filtration}, $\essa \cap B(p,t)$ is contained in the \emph{S}-filtration of $Y$ in $B(p,t)$.  Again applying our induction hypothesis on dimension, we conclude that there exists $r'>0$ such that $\essa \cap B(p,r') = \emptyset$.  This completes the proof.  \end{proof}

We can now finish the proof of Theorem \ref{support_theorem}.   Given $p \in {\rm int}(K)$ choose $r > 0$ such that $B(p,r) \cap V \subset \mbox{int}(K)$.  Taking $\Sigma = V \cap B(p,r)$ in Lemma \ref{support_in_S_filtration}, we see that $\essa \cap B(p,r)$ is contained in the \emph{S}--filtration of $V \cap B(p,r)$.  Then by Lemma \ref{support_in_boundary}, there is an $r'>0$ such that the support of $\essa$ is disjoint from $B(p,r')$.  We conclude that $\essa$ is contained in $\partial K$, as desired.


\begin{thebibliography}{99}

\bibitem{Alex-Wermer} H. Alexander and J. Wermer, {\it
Several Complex Variables and Banach Algebras}, 3rd~ed.,
Springer, New York, 1998.

\bibitem{anderson_izzo_two_manifolds} J. T. Anderson and A. J. Izzo, ``A peak point theorem for uniform algebras generated by smooth functions on a two-manifold,'' \emph{Bull.\ Lond.\ Math.\ Soc.\ }{\bf 33} (2001), 187--193.

\bibitem{anderson_izzo_smooth_manifolds} J. T. Anderson and A. J. Izzo, ``Peak point theorems for uniform algebras on smooth manifolds,'' \emph{Math. Z.} {\bf 261} (2009), 65--71.

\bibitem{anderson_izzo_varieties} J. T. Anderson and A. J. Izzo, \lq\lq A peak point theorem for uniform algebras on real-analytic varieties,\rq\rq\ {\it Math.\ Annalen\/} 364 (2016), 657--665. 

\bibitem{anderson_izzo_local} J. T. Anderson and A. J. Izzo, \lq\lq Localization for uniform algebras generated by real-analytic functions,\rq\rq\ \emph{Proc.\ Amer.\ Math.\ Soc.\ } (accepted). 

\bibitem{anderson_izzo_wermer_three_dimensions} J. T. Anderson, A. Izzo, and J. Wermer, ``Polynomial approximation on three-dimensional real-analytic submanifolds of $\mathbb{C}^n$,'' \emph{Proc.\ Amer.\ Math.\ Soc.\ }{\bf 129} (2001) no. 8, 2395--2402.

\bibitem{anderson_izzo_wermer_varieties} J. T. Anderson, A. J. Izzo, and J. Wermer, ``Polynomial approximation on real-analytic varieties in $\mathbb{C}^n$,'' \emph{Proc. Amer. Math. Soc. }{\bf 132} (2003) no. 5, 1495--1500.

\bibitem{basener} R. Basener, ``On rationally convex hulls,'' \emph{Trans. Amer. Math. Soc.} {\bf 182} (1973), 353--381.

\bibitem{bear} H. S. Bear,
\lq\lq Complex function algebras,\rq\rq\ \emph{Trans.\ Amer.\ Math.\ Soc.}  {\bf 90} (1959), 383--393.

\bibitem{boggess} A. Boggess, {\it CR Manifolds and the Tangential Cauchy-Riemann Complex\/}, 
Studies in Advanced Mathematics. CRC Press, Boca Raton, FL, 1991.


\bibitem{browder} A. Browder, \emph{Introduction to Function Algebras}, Benjamin, New York (1969).

\bibitem{Co} B. J. Cole, {\it One-Point Parts and the Peak Point Conjecture\/}, Ph.D. dissertation, Yale Univ., 1968.

\bibitem{diederich-fornaess} K. Diederich and J. Fornaess, ``Pseudoconvex domains with real-analytic boundary,'' \emph{Ann.\ of Math.} (2) {\bf 107} (1978), no. 2, 371--384.

\bibitem{freeman} M. Freeman, ``Some conditions for uniform approximation on a manifold,'' \emph{Function Algebras} (ed. F. Birtel, Scott, Foresman and Co., Chicago, 1966), 42--60.

\bibitem{gamelin} T. W. Gamelin, 
{\it Uniform Algebras}, 2nd ed., Chelsea Publishing Company, New York, NY, 1984.

\bibitem{ghosh} S. N. Ghosh, ``Isolated point theorems for uniform algebras on two- and three-manifolds,'' \emph{Proc.\ Amer.\ Math.\ Soc.} {\bf 144} (2016), 3921--3933.

\bibitem{ghosh2} S. N. Ghosh, \lq\lq An isolated point theorem for uniform algebras on three-manifolds" (in preparation).


\bibitem{izzo_survey} A. J. Izzo, ``The peak point conjecture and uniform algebras invariant under group actions,'' \emph{Function Spaces in Modern Analysis}, Contemp. Math. {\bf 547}, Amer. Math. Soc. 2011, 135--145.

\bibitem{izzo_approx_on_manifolds} A. J. Izzo, ``Uniform approximation on manifolds,'' \emph{Ann.\ of Math.} (2) {\bf 174} (2011), no.~1, 55 -- 73.

\bibitem{izzo_kalm_wold} A. J. Izzo, H. Samuelsson Kalm, and E. F. Wold, 
\lq\lq Presence or absence of analytic structure in maximal ideal spaces,"
\emph{Math.\ Ann.} {\bf 366} (2016), 459--478.

\bibitem{narasimhan} R. Narasimhan, \emph{Introduction to the Theory of Analytic Spaces}, Lecture Notes in Mathematics no. 25, Springer-Verlag, Berlin, 1966.

\bibitem{stout_book} E. L. Stout, \emph{The Theory of Uniform Algebras}, Bogden and Quigley, (1971).

\bibitem{stout_varieties} E. L. Stout, ``Holomorphic approximation on compact, holomorphically convex, real-analytic varieties,'' \emph{Proc. Amer. Math. Soc.}, {\bf 134} (2006), no. 8, 2303--2308.


\bibitem{wermer64}  J. Wermer, \lq\lq Approximation on a disk,"
{\em Math.\ Ann.\ }{\bf 155}
(1964), 331--333.

\bibitem{wermer} J. Wermer, ``Polynomially convex disks,'' \emph{Math.\ Ann.} {\bf 158} (1965), 6--10.
\end{thebibliography}
\end{document}